\documentclass{article}

\usepackage{amsmath, amsthm}
\usepackage{amssymb}
\usepackage{color}
\usepackage{tikz}
\usetikzlibrary{intersections}

\newcommand{\CC}{\mathbb{C}}

\newcommand{\RR}{\mathbb{R}}

\newcommand{\Schrodinger}{Schr\"odinger }

\newcommand{\Paivarinta}{P\"aiv\"rinta}

\newtheorem{thm}{Theorem}[section]
\newtheorem{prop}[thm]{Proposition}
\newtheorem{lemma}[thm]{Lemma}

\title{Inverse problem of electroseismic conversion. I: Inversion 
       of Maxwell's equations with internal data}
\author{Jie Chen\\ 
Department of Mathematics, Purdue University,\\ 
West Lafayette, IN 40907\\
chenjie@uw.edu
\and 
Maarten de Hoop \\
Department of Mathematics, Purdue University,\\
West Lafayette, IN 40907\\  
mdehoop@purdue.edu}
\date{\today}

\begin{document}

\maketitle

\begin{abstract}
Pride (1994, Phys. Rev. B 50 15678–96) derived the governing model of
electroseismic conversion, in which Maxwell's equations are coupled
with Biot's equations through an electrokinetic mobility
parameter. The inverse problem of electroseismic conversion was first
studied by Chen and Yang (2013, Inverse Problem 29 115006). By
following the construction of Complex Geometrical Optics (CGO)
solutions to a matrix \Schrodinger equation introduced by Ola and
Somersalo (1996, SIAM J. Appl. Math. 56 No. 4 1129-1145), we analyze
the reconstruction of conductivity, permittivity and the
electrokinetic mobility parameter in Maxwell's equations with internal
measurements, while allowing the magnetic permeability $\mu$ to be a
variable function. We show that knowledge of two internal data sets
associated with well-chosen boundary electric sources uniquely
determines these parameters. Moreover, a Lipschitz-type stability is
obtained based on the same set.
\end{abstract}

\section{Introduction}

In fluid-saturated porous media, an electrical double layer (EDL) is
formed at the interface (the pore boundaries) of the fluid and solid
rock. The fluid side of the EDL is charged with positive ions and the
solid side with negative electrons. When electric or magnetic fields
impinge on the EDL, the electrokinetic phenomenon causes movement of
the fluid relative to rock frame and thus emits seismic waves, which
can be remotely detected. This effect is named electroseismic
conversion. Conversely, a seismic wave can cause the separation of
charges and therefore generate electromagnetic fields. At the
intersection of two formation layers, the singularities in electrical,
hydraulic and mechanical properties will lead to a discontinuity in
the induced electromagnetic fields, and emit electro-magnetic waves,
which can be remotely observed. This effect is named
seismoelectric conversion. In fact, these two ways of conversion
always happen simultaneously.

The wave propagation in fluid-saturated porous media was studied by
Biot \cite{Biot1956,Biot1956_2}. In Biot's theory, in addition to the
conventional compressional and shear waves, a compressional slow wave
appears. The first experimental observation of this slow wave was
obtained by Plona \cite{Plona1980}. The predication of the slow wave
has been quantitatively confirmed by Pride \cite{Pride1994}. Based on 
Biot's theory, Pride \cite{Pride1994} also developed
the governing equations of electroseismic conversion, which are
\begin{eqnarray}
	&\nabla\wedge E = i\omega\mu H,&
			\label{Maxwell1}\\
	&\nabla\wedge H = (\sigma - i\epsilon\omega)E
			+ L(-\nabla p +\omega^2\rho_f u) + J_s,&
			\label{Maxwell2}\\
	&-\omega^2(\rho u + \rho_f w) = \nabla\cdot \tau,&
			\label{Biot1}\\
	&-i\omega w = LE + \frac{\kappa}{\eta}(-\nabla p + \omega^2\rho_f u),&
			\label{Biot2}\\
	&\tau = (\lambda\nabla\cdot u + C\nabla\cdot w)I
			+ G(\nabla u + \nabla u^T),&
			\label{Biot3}\\
	&-p = C\nabla\cdot u + M\nabla\cdot w,&
			\label{Biot4}
\end{eqnarray}
where the first two are Maxwell's equations, and the remaining ones
are Biot's equations. The notation is as follows:
\begin{itemize}
	\item[$E$]        electric field,
	\item[$H$]        magnetizing field or magnetic field intensity,
	\item[$\omega$]   seismic wave frequency,
	\item[$\sigma$]   conductivity,
	\item[$\epsilon$] dielectric constant or relative permittivity,
	\item[$\mu$]      magnetic permeability,
	\item[$J_s$]      source current,
	\item[$p$]        pore pressure,
	\item[$\rho_f$]   density of pore fluid,
	\item[$L$]        electro-kinetic mobility parameter,
	\item[$\kappa$]   fluid flow permeability,
	\item[$u$]        solid displacement,
	\item[$w$]        relative fluid displacement, 
	\item[$\tau$]     bulk stress tensor,
	\item[$\eta$]     viscosity of pore fluid,
	\item[$\lambda, G$] Lam$\acute{\mathrm{e}}$ parameters of elasticity,
	\item[$C, M$]     Biot moduli parameters.
\end{itemize}
Some basic properties of the coupling effect were studied by Pride and
Haartsen in \cite{Pride1996}.

Under the assumption that the coupling is weak, depending on the
source, we focus on one way coupling and ignore the second way which
is weaker as it is caused by induced waves. In the case of a seismic
boundary source generating waves, seismoelectric conversion will
dominate and the governing equations are given as above but with the
term $LE$ removed from \eqref{Biot2}. Thompson and Gist
\cite{Thom1993} made the first field measurements in 1993
demonstrating the application of seismoelectric conversion as a survey
tool. Zhu \emph{et al.} \cite{Zhu2003,Zhu2005} performed a series of
laboratory experiments in model wells and studied the application of
seismoelectric conversion as a bore-hole logging tool as well as a
cross-hole logging tool. In 2011, the laboratory experiments by
Schakel \cite{Schakel2011} also indicated the promising application of
seismoelectric conversion for subsurface exploration.

On the other hand, starting with electric field boundary sources,
electroseismic conversion is dominant and seismoelectric conversion is
negligible, in which case the governing model is as above but with the
term $L (-\nabla p + \omega^2 \rho_f u)$ removed in
\eqref{Maxwell2}. In 2005, White \cite{White2005} established a
forward model for the electroseismic method by combining Maxwell's
equations and the elastic wave equation, while the initial amplitude
of elastic waves is calculated according to Pride’s equations with a
high-frequency asymptotic theory. In 2007, Thompson \emph{et al.} at
ExxonMobil \cite{Thom2007} presented results from field tests of
electroseismic conversion in Texas and Canada. By using specially
designed boundary electric current waveforms at dominant frequencies
of 8Hz, 18Hz and 25Hz, their tests over gas sands and carbonate oil
reservoirs succeeded in delineating known hydrocarbon accumulations
from depths up to 1500m, which suggested applicability of
electroseismic conversion at significant depths. The laboratory
experiments by Schakel \cite{Schakel2011} demonstrated that higher
frequencies in electric sources lead to smaller aptitude in induced
seismic
waves. The critical advantage of exploiting electroseismic conversion
is that the experiments provide certain internal data for the inverse
problem for the (time-harmonic) Maxwell's equations and, as we will
show, lead to well-posedness of this problem. In this paper, we
analyze the inverse problem of electroseismic conversion, which was
first studied mathematically by Chen and Yang \cite{Chen2013}, with
internal data given by $\Sigma := L E$. The analysis of recovering
these internal data from boundary measurements using Biot's equations
will be presented in a separate paper.

Let $\Omega$ be an open bounded subset of $\RR^3$ with smooth boundary
$\partial\Omega$. Let $D=-i\nabla$. The time-harmonic Maxwell's
equations are given by
\begin{equation}\label{eq3.1}
    \left\{\begin{array}{l}
        D\wedge H + \omega\gamma E = 0 ,\\
        D\wedge E - \omega\mu H = 0 ,
    \end{array}\right.
\end{equation}
where $\omega > 0$ is a fixed angular frequency and $\gamma =
\epsilon + i\sigma/\omega$. We assume that $\mu, \gamma, \epsilon, 
\sigma \in H^s(\RR^3)$, $s>3/2$, satisfy
\begin{equation}
    \label{eq2.0}
	\mu > 0  \;\mathrm{and}\;  \gamma \neq 0 \;\mathrm{in}\; \RR^3.
\end{equation} 
By substitution, we have the curl-curl form of Maxwell's equations
\begin{equation}
\label{eq2.1} D \wedge \mu^{-1} D \wedge E + \omega^2\gamma E = 0.
\end{equation}
The boundary source is expressed in terms of the boundary tangential
components of the electric field,
\begin{equation}
\label{eq2.2} G:=tE \quad\mathrm{on}\;\partial\Omega,
\end{equation}
where $tE$ is the tangent component of $E$. We also assume that
the internal data of the form
\begin{equation}\notag 
\label{eq2.3}	\Sigma := L E \quad\mathrm{in}\;\Omega
\end{equation}
are given. We define the forward operator $\Lambda_G$ by,
\begin{equation}\notag 
\label{eq2.4} (L,\gamma) \to \Sigma = \Lambda_G(L,\gamma).
\end{equation}
The problem studied in this paper is the inversion of operator
$\Lambda_G$. Precisely, given properly chosen the boundary values of
the electrical field, $G$, can we recover the coupling coefficient $L$
and the complex parameter $\gamma$ from the internal data $\Sigma$?

Inverse problems of Maxwell's equations with other types of internal
data are also studied in \cite{Bal2013,Bal2014}. 
One reconstruction methodology of the inverse problem with internal
data, inspired by Bal and Uhlmann's study of photo-acoustic tomography
(PAT) \cite{Bal2010}, primarily consists of converting the governing
equation to a transport equation and constructing Complex Geometrical
Optics (CGO) solutions to the governing equations. The unique and
stable solvability of the transport equation relies on the estimate of
the vector field in the transport equation associated with the
explicitly constructed CGO solutions. The same methodology was
followed by Chan and Yang \cite{Chen2013}, converting Maxwell's
equations to a transport equation while constructing CGO solutions to
Maxwell equations. They proved the uniqueness and stability of the
reconstruction of the coupling coefficient $L$ and conductivity
$\sigma$ in the case of a constant magnetic permeability $\mu$. Here,
we consider the general case with variable $\mu$ including the
recovery of relative permittivity.

Our construction of CGO solutions for variable $\mu$ employs the idea
of converting Maxwell's equations to a matrix \Schrodinger
equation. The matrix \Schrodinger equation formulation was first
developed by Ola and Somersalo \cite{Ola1993} to study the inverse boundary value
problem in electromagnetics; the relevant analysis was simplified in
\cite{Ola1996}.

This paper is organized as follows: in Section \ref{se:Schrodinger},
we introduce the matrix \Schrodinger equation the CGO solutions of
which are constructed in Section \ref{se:CGO}. Our main theorems are
stated and proven in Section \ref{se:Maxwell}, while we derive the
transport equation and prove the uniqueness result in Section
\ref{se:unique}. We prove the stability results in Sections
\ref{se:stab} and \ref{se:stab2}. In Section \ref{se:time}, we address
the temporal behavior of CGO solutions.

\section{Matrix \Schrodinger equation}\label{se:Schrodinger}

If $\mu\equiv\mu_0$ is constant, the curl-curl form of Maxwell's
equation in \eqref{eq2.1} is simply given by
\begin{equation}
	\label{eq:sch1} D\wedge D\wedge E + \omega^2\gamma\mu_0 E = 0,
\end{equation}
the CGO solutions of which were constructed by Colton
\cite{Colton1992} and extended to high order Sobolev spaces by Chen
and Yang \cite{Chen2013}. However, their construction method fails for
non-constant $\mu$. In the present work, we convert the first-order
system of Maxwell's equations to a matrix \Schrodinger equation and
construct corresponding CGO solutions following
\cite{Ola1996,Salo2009}.

Let $\mu, \epsilon,\sigma\in H^s(\RR^3)$, $s> \frac{3}{2}$, satisfy \eqref{eq2.0}. Taking the divergence of \eqref{eq3.1} gives
\begin{equation}\label{eq3.2}
\left\{\begin{array}{l}
D\cdot(\gamma E) = 0,\\
D\cdot(\mu H) = 0.
\end{array}\right.
\end{equation}
Let $\alpha = \log\gamma$ and $\beta=\log\mu$. By combining \eqref{eq3.1} and \eqref{eq3.2}, we get
\begin{equation}\label{eq3.3}
\left\{\begin{array}{rcl}
	D\cdot E + D\alpha\cdot E &=& 0,\\
	-D\wedge E + \omega\mu H &=& 0,\\
	D\cdot H + D\beta\cdot H &=& 0,\\
	D\wedge H + \omega\gamma E &=& 0.
\end{array}\right.
\end{equation}
The above system contains 8 equations and 6 unknown components in $E$
and $H$. We write \eqref{eq3.3} as a $8\times 8$ matrix system,
\begin{equation}\label{eq3.4}
\left[
\left(\begin{array}{cccc}
	* & 0 & * & D\cdot\\
	* & 0 & * & -D\wedge\\
	* & D\cdot & * & 0\\
	* & D\wedge & * & 0
\end{array}\right)
+
\left(\begin{array}{cccc}
	* & 0 & * & D\alpha\cdot\\
	* & \omega\mu I_3 & * & 0\\
	* & D\beta\cdot & * & 0\\
	* & 0 & * & \omega\gamma I_3
\end{array}\right)
\right]
\left(\begin{array}{c}
0\\H\\0\\E
\end{array}\right)
=0.
\end{equation}
We define
\begin{equation}\notag 
P_+(D) = 
\left(\begin{array}{cc}
	0 & D\cdot\\
	D & D\wedge
\end{array}\right),\quad 
P_-(D) = 
\left(\begin{array}{cc}
	0 & D\cdot\\
	D & -D\wedge
\end{array}\right),
\end{equation}
which satisfy
\begin{eqnarray}
 \notag && P_+(D)P_-(D) =  P_-(D)P_+(D) = -\Delta I_4,\\
 \notag && P_+(D)^* = P_-(D),\quad P_-(D)^* = P_+(D).
\end{eqnarray}
We also define
\begin{equation}\notag 
P^\mp(D) = \left(\begin{array}{cc}
	0 & P_-(D)\\
	P_+(D) & 0
\end{array}\right).
\end{equation}
Then $P^\mp(D)P^\mp(D) = -\Delta I_8$ and $P^\mp(D)^*=P^\mp(D)$. We denote
\begin{equation}\notag 
P^\mp(a,b) = \left(\begin{array}{cc}
	0 & P_-(b)\\
	P_+(a) & 0
\end{array}\right),\quad
P^\pm(a,b) = \left(\begin{array}{cc}
	0 & P_+(b)\\
	P_-(a) & 0
\end{array}\right),
\end{equation}
where $a,b\in\CC^3$, and denote
\begin{equation}\notag 
\mathrm{diag}(A,B) = \left(\begin{array}{cc}
	A & 0\\
	0 & B
\end{array}\right).
\end{equation}
Here $A,B$ are $4\times 4$ diagonal complex-valued matrices.

We now introduce scalar fields $\Phi$ and $\Psi$, and $X = (\Phi,\; H,\; \Psi,\; E)^t$, so that the matrix system in \eqref{eq3.4} attains the form
\begin{equation}\notag 
	(P^\mp(D) + V_{\mu,\gamma})X = 0 \quad \mathrm{in}\;\Omega,
\end{equation}
where
\begin{equation}\notag 
V_{\mu,\gamma} = \left(\begin{array}{cccc}
	\omega\mu & 0 & 0 & D\alpha\cdot \\
	0 & \omega\mu I_3 & D\alpha & 0 \\
	0 & D\beta\cdot & \omega\gamma & 0\\
	D\beta & 0 & 0 & \omega\gamma I_3
\end{array}\right).
\end{equation}
Solutions to this system with $\Phi=\Psi=0$ correspond to solutions of the original Maxwell system. By changing variables so that
\begin{equation}\label{eq3.8}
	Y = \mathrm{diag}(\mu^{1/2},\gamma^{1/2})X,
\end{equation}
we have
\begin{equation}
\label{eq3.5}
	(P^\mp(D) + V_{\mu,\gamma})X=0\quad \Leftrightarrow\quad (P^\mp(D) + W_{\mu,\gamma})Y=0,
\end{equation}
where $\kappa=\omega(\gamma\mu)^{1/2}$ and 
\begin{equation}\notag 
W_{\mu,\gamma} = \left(\begin{array}{cc}
	\kappa I_4 & \frac{1}{2}P_+(D\alpha)\;\\
	\frac{1}{2}P_-(D\beta) \; & \kappa I_4
\end{array}\right)
=\kappa I_8 + \frac{1}{2}P^\pm(D\beta,D\alpha).
\end{equation}
A direct calculation leads to the following

\begin{lemma}[\cite{Salo2009}]\label{thm:operator_PP}
One has
\begin{eqnarray}
	\notag && (P^\mp(D) + W_{\mu,\gamma})(P^\mp(D) - W_{\mu,\gamma}^t) = -\Delta I_8 + \tilde{Q}_{\mu,\gamma},\\
	\notag && (P^\mp(D) - W_{\mu,\gamma}^t)(P^\mp(D) + W_{\mu,\gamma}) = -\Delta I_8 + \tilde{Q}_{\mu,\gamma}'.
\end{eqnarray}
Here, the matrix potentials, $\tilde{Q}_{\mu,\gamma}$ and
$\tilde{Q}'_{\mu,\gamma}$, are given by
\begin{eqnarray}
	\notag \tilde{Q}_{\mu,\gamma} &=& \frac{1}{2}
		\left(\begin{array}{c|c}
			\begin{array}{cc}
				\Delta\alpha & 0 \\
				0 & 2\nabla\nabla\alpha-\Delta\alpha I_3
			\end{array} & 0\\
			\hline
			0 & \begin{array}{cc}
					\Delta\beta & 0 \\
					0 & 2\nabla\nabla\beta-\Delta\beta I_3
				\end{array}
	\end{array}\right)\\
	\notag && - \left(\begin{array}{c|c}
		(\kappa^2 + \frac{1}{4}(D\alpha)^2)I_4 & 
			\begin{array}{cc}
				0 & D\kappa \\
				2D\kappa & 0
			\end{array} \\
		\hline
		\begin{array}{cc}
				0 & 0 \\
				0 & -2D\kappa\wedge
		\end{array} & (\kappa^2 + \frac{1}{4}(D\beta)^2)I_4
	\end{array}\right),
\end{eqnarray}
and 
\begin{eqnarray}
	\notag \tilde{Q}_{\mu,\gamma} &=& \frac{1}{2}
		\left(\begin{array}{c|c}
			\begin{array}{cc}
				\Delta\beta & 0 \\
				0 & 2\nabla\nabla\beta-\Delta\beta I_3
			\end{array} & 0\\
			\hline
			0 & \begin{array}{cc}
					\Delta\alpha & 0 \\
					0 & 2\nabla\nabla\alpha-\Delta\alpha I_3
				\end{array}
	\end{array}\right)\\
	\notag && - \left(\begin{array}{c|c}
		(\kappa^2 + \frac{1}{4}(D\beta)^2)I_4 & 
			\begin{array}{cc}
				0 & D\kappa \\
				2D\kappa & 0
			\end{array} \\
		\hline
		\begin{array}{cc}
				0 & 0 \\
				0 & -2D\kappa\wedge
		\end{array} & (\kappa^2 + \frac{1}{4}(D\alpha)^2)I_4
	\end{array}\right),
\end{eqnarray}
with $\nabla\nabla f = (\partial^2_{x_j,x_k}f)^3_{j,k=1}$.
\end{lemma}

We can extend $\mu$ and $\gamma$ to $\RR^3$ so
that for some nonzero constants $\mu_0$ and $\gamma_0$,
$\mu - \mu_0$ and $\gamma - \gamma_0$ are compactly supported. We also
assume that $Z$ is the solution to
\begin{equation}
\label{eq3.6}
	(P^\mp(D) - W^t_{\mu,\gamma})Z = Y.
\end{equation}
Lemma \ref{thm:operator_PP} then implies that $Z$ solves the matrix \Schrodinger equation
\begin{equation}\label{eq3.7}
	(-(\Delta + k^2)I_8 + Q_{\mu,\gamma})Z=0,
\end{equation}
where $k=\omega(\mu_0\gamma_0)^{1/2}$ and $Q_{\mu,\gamma} = k^2 I_8 +
\tilde{Q}_{\mu,\gamma}$. It follows that $Q_{\mu,\gamma}$ is
compactly supported.

In the following section, we will construct solutions to
\eqref{eq3.7} and thus solutions to Maxwell's equations according to
\eqref{eq3.8} and \eqref{eq3.6}.

\section{Complex Geometrical Optics(CGO) solutions}\label{se:CGO}

Sylvester and Uhlmann \cite{Uhlmann1987} constructed CGO solutions to
the scalar \Schrodinger equation. Ola and Somersalo \cite{Ola1996}
followed the Sylvester-Uhlmann method to construct CGO solutions to
the matrix \Schrodinger equation given in \eqref{eq3.7}. The CGO
solutions by Ola and Somersalo are in a weighted $L^2$ space. In the
present work, we apply the Sylvester-Uhlmann method to construct CGO
solutions in higher order Sobolev spaces.

We first introduce some notation. Let the space $L^2_\delta$ for
$\delta\in\RR$ be the completion of $C^\infty_0(\RR^3)$ with respect
to the norm $\|\cdot\|_{L^2_{\delta}}$ defined by
\begin{equation}
	\label{eq3.9}\notag 
	\|u\|_{L^2_\delta} = \left(\int_{\RR^3} \langle x\rangle^{2\delta}|u|^2dx\right)^{1/2},\quad \langle x\rangle = (1+|x|^2)^{1/2}.
\end{equation}
We also define the space $H^s_\delta$ for $s>0$ as the completion of $C^\infty_0(\RR^3)$ with respect to the norm $\|\cdot\|_{H^s_\delta}$ defined by
\begin{equation}\notag 
	\label{eq3.10} \|u\|_{H^s_\delta} = \left(\int_{\RR^3} \langle x\rangle^{2\delta}|(I-\Delta)^{\frac{s}{2}}u|^2 dx\right)^{1/2}.
\end{equation}
Here $(I-\Delta)^{\frac{s}{2}}u$ is defined as the inverse Fourier transform of $\langle\xi\rangle\hat{u}(\xi)$, where $\hat{u}(\xi)$ is the Fourier transform of $u(x)$. When $\delta=0$, this is the standard Sobolev space $H^s(\RR^3)$ of order $s$.

It is proven in \cite{Uhlmann1987} that for $|\zeta|\geq c >0$ and $v\in L^2_{\delta+1}$ with $-1<\delta<0$, equation
\begin{equation}
	\label{eq3.11} (\Delta-2\zeta\cdot D)u=v, 
\end{equation}
admits a unique weak solution $u\in L^2_\delta$ with
\begin{equation}
	\label{eq3.12} \|u\|_{L^2_\delta}\leq C \frac{\|v\|_{L^2_{\delta+1}}}{|\zeta|},
\end{equation}
for some constant $C=C(\delta,c)$. We note that $(\Delta + 2i\zeta\cdot D)$ and $(I-\Delta)^s$ are constant coefficient operators and hence commute. We deduce that for $v\in H^s_{\delta+1}$, for $s\geq 0$, \eqref{eq3.11} admits a unique solution $u\in H^s_\delta$ with
\begin{equation}
	\label{eq3.13} \|u\|_{H^s_\delta}\leq C(\delta,c) \frac{\|v\|_{H^s_{\delta+1}}}{|\zeta|}.
\end{equation}
One defines the integral operator $G_\zeta:H^s_{\delta+1}(\RR^3) \rightarrow H^s_\delta(\RR^3)$ by
\begin{equation}\notag
	\label{eq3.14} G_\zeta(v):=\mathcal{F}^{-1}\left(\frac{\hat{v}}{|\xi|^2 + 2\zeta\cdot\xi}\right),
\end{equation}
where $\mathcal{F}^{-1}$ is the inverse Fourier transform. Clearly, for $|\zeta|\geq c>0$, $G_\zeta$ is the inverse operator of $(\Delta + 2\zeta\cdot D)$ and $G_\zeta$ is bounded by
\begin{equation}
	\label{eq3.15} \|G_\zeta\|\leq\frac{C}{|\zeta|},
\end{equation}
for some constant $C=C(\delta,c)$.

As in \cite{Salo2009}, we choose
$\zeta\in\CC^3$ such that $\zeta\cdot\zeta = k^2$ and $|\zeta|$
large. Compared to \cite{Ola1996,Salo2009}, the following proposition
extends the construction of CGO solutions to $H^s_\delta(\Omega)$,
while requiring more regularities in $\mu,\gamma$.

\begin{prop}
Let $s\geq \frac{3}{2}$ and $-1<\delta<0$. Assume that $\mu,\gamma\in
H^{s+2}(\Omega)$ and $\mu,\gamma$ are constant outside a compact
set. There exists a CGO solution in $H^s_\delta(\Omega)$ of the form
\begin{equation}\notag 
	Z = e^{i\zeta\cdot x}(Z_0 + Z_r)
\end{equation}
to equation \eqref{eq3.7}, such that $Z_0\in \CC^3$ is a constant vector and
\begin{equation}\notag 
	\|Z_r\|_{H^s_\delta(\Omega)}\leq \mathcal{O}(|\zeta|^{-1}).
\end{equation}
\end{prop}

\begin{proof}
Since $\zeta\cdot\zeta = k^2$, we have
\begin{align}
	\label{eq3.16} 0  = & e^{-i\zeta\cdot x}(-(\Delta+k^2)I_8 + Q_{\mu,\gamma}) e^{i\zeta\cdot x}(Z_0 + Z_r)\\
	\notag = & (-\Delta + 2\zeta\cdot D)Z_r + Q_{\mu,\gamma}(Z_0 + Z_r),
\end{align}
or
\begin{equation}
	\label{eq3.18} (-\Delta+2\zeta\cdot D)Z_r + Q_{\mu,\gamma}Z_r = -Q_{\mu,\gamma}Z_0.
\end{equation}
When $s\geq \frac{3}{2}$, $H^s(\RR^3)$ is an algebra. By the
assumption, $Q_{\mu,\gamma}\in H^s(\Omega)$ is compactly supported and
thus multiplication by $Q_{\mu,\gamma}$ is a bounded operator mapping
$H^s_{\delta}(\RR^3)$ to $H^s_{\delta+1}(\RR^3)$. In particular,
$Q_{\mu,\gamma}Z_0\in H^s_{\delta+1}(\RR^3)$ for any constant
$Z_0$. The estimate in \eqref{eq3.15} implies that $I +
G_\zeta(Q_{\mu,\gamma}\cdot)$ is well defined and is an invertible mapping
$H^s_{\delta}(\RR^3)$ to $H^s_\delta(\RR^3)$ for $|\zeta|$
sufficiently large. We can then define
\begin{equation}
	\label{eq3.19} Z_r = -(I + G_\zeta(Q_{\mu,\gamma}\cdot))^{-1}G_\zeta(Q_{\mu,\gamma}Z_0) \in H^s_\delta(\RR^3).
\end{equation}
It is immediate that
\begin{equation}
	\label{eq3.17} Z_r+G_\zeta(Q_{\mu,\gamma}Z_r) = -G_\zeta(Q_{\mu,\gamma}Z_0).
\end{equation}
$Z_r$ also satisfies \eqref{eq3.16} and \eqref{eq3.18}. Thus $e^{i\zeta\cdot x}(Z_0 + Z_r)\in H^s_\delta(\RR^3)$ is a solution to \eqref{eq3.7}. We also see from \eqref{eq3.19} that
\begin{equation}
	\|Z_r\|_{H^s_\delta} \leq \mathcal{O}(|\zeta|^{-1})\|Q_{\mu,\gamma}Z_0\|_{H^s_{\delta+1}} \leq \mathcal{O}(|\zeta|^{-1}).
\end{equation}
\end{proof}

With CGO solutions to the matrix \Schrodinger equation, we can now construct solutions to the original Maxwell equations. We define $Y$ by \eqref{eq3.6} and thus $Y$ satisfies \eqref{eq3.5}. Note that the solution $Y$ can be written as
\begin{equation}
	Y = e^{i\zeta\cdot x} (Y_0 + Y_r),
\end{equation}
where
\begin{equation}
	Y_0 = P^\mp(\zeta)Z_0,\quad Y_r = (P^\mp(\zeta) - W^t_{\mu,\gamma})Z_r - W^t_{\mu,\gamma}Z_0,
\end{equation}
satisfying 
\begin{equation}
	\|Y_0/|\zeta|\|_{L^2} = \mathcal{O}(1),\quad \|Y_r/|\zeta|\|_{H^{s-1}_{\delta}} = \mathcal{O}(|\zeta|^{-1}).
\end{equation}
Here, $Y_0$ is a constant vector depending on the choice of $\zeta$
and $Z_0$. We denote the components of $Y$ by
\begin{equation}\notag 
	Y = (Y^\Phi,\; (Y^H)^t,\; Y^\Psi,\; (Y^E)^t)^t.
\end{equation}
We recall that \eqref{eq3.5} is equivalent to the original Maxwell equations only if $Y^\Phi=Y^{\Psi}=0$. This condition can be satisfied with properly chosen $Z_0$ according to the following lemma, which is cited from \cite{Salo2009} without proof.

\begin{lemma}[\cite{Salo2009}]\label{thm:lemma3}
If
\begin{equation}\notag
	((P^\mp(\zeta)-k)Z_0)^\Phi = ((P^\mp(\zeta)-k)Z_0)^\Psi = 0,
\end{equation}
then we have 
\begin{equation}\notag
	Y^t = (0\;\; (Y^H)^t\;\; 0\;\; (Y^E)^t)
\end{equation}
for $|\zeta|$ sufficiently large.
\end{lemma}

\begin{prop}
The Maxwell equations have a solution of the form 
\begin{equation}\notag
	X=e^{i\zeta\cdot x}(X_0+X_r) 
\end{equation}
satisfying
\begin{equation}\notag
	\|X_0/|\zeta|\|_{H^{s+2}_{\delta}} = \mathcal{O}(1),\quad \|X_1/|\zeta|\|_{H^{s-1}_{\delta}} = \mathcal{O}(|\zeta|^{-1}).
\end{equation}
\end{prop}

The proof of the above proposition follows directly from \eqref{eq3.8}
and Lemma \ref{thm:lemma3}. An explicit choice of $Z_0$ satisfying
Lemma \ref{thm:lemma3} is given by
\begin{equation}\notag
Z_0 = \frac{1}{|\zeta|} 
\left(\begin{array}{c}
	\zeta\cdot a\\
	kb\\
	\zeta\cdot b\\
	ka
\end{array}\right),
\end{equation}
where $a,b$ are any vectors in $\CC^3$. It follows that
\begin{equation}\notag
X_0 =  \mathrm{diag}(\mu^{-1/2},\gamma^{-1/2}) P^\mp(\zeta)Z_0 = \frac{1}{|\zeta|} 
\left(\begin{array}{l}
	ka\cdot\zeta/\mu^{-1/2}\\
	((b\cdot\zeta)\zeta - k\zeta\times a)/\mu^{-1/2}\\
	kb\cdot\zeta/\gamma^{-1/2}\\
	((a\cdot \zeta)\zeta + k\zeta\times b)/\gamma^{-1/2}
\end{array}\right).
\end{equation}
We choose $b\in\CC^3$ such that $|b|=\mathcal{O}(|\zeta|)$, and define a CGO solution as
\begin{equation}
\label{eq:CGOE}
\check{E} = \frac{1}{|\zeta|}X^E = e^{i\zeta\cdot x}(\check{E}_0 + \check{E}_r),
\end{equation}
where
\begin{equation}
	\check{E}_0 = \frac{\sqrt{\gamma}}{|\zeta|^2}((a\cdot\zeta)\zeta + k\zeta\times b) \in H^{s+2}(\Omega)
\end{equation}
and
\begin{equation}
	\|\check{E}_0\|_{H^{s+2}(\Omega)} = \mathcal{O}(1), \quad \|\check{E}_r\|_{H^{s-1}_\delta(\Omega)} = \mathcal{O}(|\zeta|^{-1}).
\end{equation}

Compared to the CGO solutions in \cite{Chen2013}, our new construction
applies to variable $\mu$, while the required regularity of
parameters $\mu,\gamma$ is one order higher. Here and below, we use a
check sign, for example $\check{E}$, to indicate corresponding
variables computed from the CGO solutions.

\section{Inversion of Maxwell's equations with internal data}\label{se:Maxwell}

In this section, we state and prove the uniqueness and stability results of
the inverse problem of Maxwell's equations with internal data. In addition to Sobolev spaces $H^s(\Omega)$, we also work on continuous function spaces $C^d(\Omega)$. Here $s$ and $d$ are related by $s = \frac{3}{2}+d+2+\iota$ for some small
$\iota>0$. We assume that $\mu \in H^{\frac{3}{2}+d+4+\iota}(\Omega)$
is known and define $\mathcal{M}$ to be the set of coefficients as
\begin{multline}
    \label{eq2.5} \mathcal{M}:=\{(L,\gamma)\in C^{d+1}(\overline{\Omega})\times H^{\frac{3}{2}+d+4+\iota}(\Omega) : d\geq 3,\; \iota>0\;\mathrm{is\; small,\; and}\\
    0 \mathrm{\;is \;not \; an\; eigenvalue\; of \;}D\wedge(\mu^{-1}D\wedge \cdot) + \omega^2\gamma\}.
\end{multline}
Our proof makes use of CGO solutions, constructed in previous section,
of the form \eqref{eq:CGOE}.  The trace of the CGO solutions in
$H^{d+3+\iota}(\partial\Omega)$ describes how we should control the
boundary values of electric fields. We note that the dominant term,
$e^{\mathrm{i}\zeta\cdot x}\check E_0$, of a CGO solution is
characterized by parameters $\zeta, a, b$ and the coefficient
$\gamma$. Therefore, for a particular choice of $\zeta, a$ and $b$ and
$\gamma$ known on the boundary of the domain, the restriction of
$e^{\mathrm{i}\zeta\cdot x}\check E_0$ to the boundary gives an
explicit form of the source distribution. In our inversion method,
we construct two CGO
solutions, according to which we define two electric sources, denoted
by $(G_1,G_2)$. The following theorems
state that, with electric sources $(G_1,G_2)$,
we can invert $\Lambda_G$ uniquely and stably.

\begin{thm}\label{thm:unique}
Let $d\geq 3$. Let $\Omega$ be an open bounded subset of $\RR^3$ with
boundary $\partial\Omega$ of class $C^d$. Let $(L,\gamma)$ and
$(\tilde{L},\tilde{\gamma})$ be two elements in $\mathcal{M}$ with
$L|_{\partial\Omega} = \tilde{L}|_{\partial\Omega}$. Let
$\Sigma:=(\Sigma_1,\Sigma_2)$ and
$\tilde{\Sigma}:=(\tilde{\Sigma}_1,\tilde{\Sigma}_2)$ be two sets of
internal data on $\Omega$ for coefficients $(L,\gamma)$ and
$(\tilde{L},\tilde{\gamma})$, respectively, and with boundary values
$G:=(G_1,G_2)$.

There is a non-empty open set of $G$ in
$(H^{d+3+\iota}(\partial\Omega))^2$, defined as a neighborhood of
the trace of CGO solutions in \eqref{eq:CGOE}, for small
$\iota>0$, such that if $\Sigma_j=\tilde{\Sigma}_j$ in
$C^{d+1}(\overline\Omega)$, $j=1,2$, we have $(L,\gamma) =
(\tilde{L},\tilde{\gamma})$.
\end{thm}

From the proof it will become apparent that the implication $L =
\tilde{L}$ requires that $d \geq 1$ and that the implication $\gamma =
\tilde{\gamma}$ requires that $d \geq 3$. Here and in the following,
we shall abuse the notation and use $C^{d}(\overline{\Omega})$ to
denote either set of complex-valued functions or set of vector-valued
functions the elements of which have up to $d$-th order continuous
derivatives. The function space $(H^{d+3+\iota}(\partial\Omega))^2$ is
an abbreviation of the product space
$H^{d+3+\iota}(\partial\Omega)\times H^{d+3+\iota}(\partial\Omega))$.

The proof of the uniqueness theorem establishes an explicit
reconstruction. Following this reconstruction we also prove a
Lipschitz stability result. To consider the stability of the
reconstruction, we need to restrict it to a subset of $\Omega$. Let
$\zeta_0$ be a constant unit vector close to $\zeta_1/|\zeta_1|$,
where $\zeta_1$ occurs in the CGO solution. Explicit choice of 
$\zeta_0$ and $\zeta_1$ is given in \eqref{eq:zeta} and 
\eqref{eq:zeta1}. Define $\Omega_1$
to be the subset of $\Omega$ obtained by removing a neighborhood of
every point $x_0\in\partial\Omega$ such that $n(x_0)\cdot\zeta_0=0$,
where $n(x_0)$ is the outward normal of $\partial\Omega$ at $x_0$.

\begin{thm}\label{thm:stability1}
Let $d\geq 3$. Let $\Omega$ be convex with boundary $\partial\Omega$
of class $C^d$ and $\Omega_1$ be defined as above. Let $(L,\gamma)$
and $(\tilde{L},\tilde{\gamma})$ be two elements in $\mathcal{M}$ with
$L|_{\partial\Omega} = \tilde{L}|_{\partial\Omega}$. Let
$\Sigma:=(\Sigma_1,\Sigma_2)$ and
$\tilde{\Sigma}:=(\tilde{\Sigma}_1,\tilde{\Sigma}_2)$ be two sets of
internal data on $\Omega$ for coefficients $(L,\gamma)$ and
$(\tilde{L},\tilde{\gamma})$, respectively, and with boundary values
$G:=(G_1,G_2)$.

There is a non-empty open set of $G$ in
$(H^{d+3+\iota}(\partial\Omega))^2$, defined as a neighborhood of
the trace of CGO solutions in \eqref{eq:CGOE}, for small
$\iota>0$ such that
\begin{equation}
\label{eq2.6}	\|L-\tilde{L}\|_{C^{d-1}(\overline{\Omega}_1)} + \|\gamma-\tilde{\gamma}\|_{C^{d-3}(\overline{\Omega}_1)} \leq C\|\Sigma-\tilde{\Sigma}\|_{(C^{d+1}(\overline{\Omega}_1))^2},
\end{equation}
for some constant $C$.
\end{thm}

We note that the constant $C$ is in the order of $\mathcal{O}(|\zeta_1|^{-1}(\|\Sigma\|_{C^{d+1}} + \|\tilde\Sigma\|_{C^{d+1}}))$, which indicates large $|\zeta_1|$ leading to better estimate.

If we have at least 6 sets of complex measurements, we can prove the
same stability results with the subset requirement removed.

\begin{thm}\label{thm:stability2}
Let $d\geq 3$. Let $\Omega$ be convex with boundary $\partial\Omega$
of class $C^d$. Let $(L,\gamma)$ and $(\tilde{L},\tilde{\gamma})$ be
two elements in $\mathcal{M}$ with $L|_{\partial\Omega} =
\tilde{L}|_{\partial\Omega}$. For $j=1,2,3$, let
$\Sigma:=(\Sigma_1^j,\Sigma_2^j)$ and
$\tilde{\Sigma}:=(\tilde{\Sigma}_1^j,\tilde{\Sigma}_2^j)$ be two sets
of internal data on $\Omega$ for coefficients $(L,\gamma)$ and
$(\tilde{L},\tilde{\gamma})$, respectively, and with boundary values
$G:=(G_1^j,G_2^j)$.

Then there is a non-empty open set of $G$ in
$(H^{d+3+\iota}(\partial\Omega))^6$, defined as a neighborhood of
the trace of CGO solutions in \eqref{eq:CGOE}, for small
$\iota>0$, such that
\begin{equation}
\label{eq2.7}	\|L-\tilde{L}\|_{C^{d-1}(\overline{\Omega})} + \|\gamma-\tilde{\gamma}\|_{C^{d-3}(\overline{\Omega})} \leq C\|\Sigma-\tilde{\Sigma}\|_{(C^{d+1}(\overline{\Omega}))^6},
\end{equation}
for some constant $C$.
\end{thm}

Our proof relies on the unique solvability of a transport equation,
the vector field of which can be controlled by properly constructed
CGO solutions of Maxwell's equations.

\subsection{Construction of vector fields and uniqueness result}
\label{se:unique}

We denote the curl-curl form of the Maxwell's equation by
\begin{equation}\label{eq:curl}\notag
	F[E]:=\nabla\times( \mu^{-1}\nabla\times E) - \omega^2\gamma E = 0.
\end{equation}
We let $E_1,E_2$ be two solutions to this equation, then
\begin{eqnarray}
	0 &=& F[E_1]\cdot E_2 - F[E_2]\cdot E_1 \notag\\
	\!\!\!&=&\!\!\! \nabla \mu^{-1}\times\nabla\times E_1\cdot E_2
            - \nabla \mu^{-1}\times\nabla\times E_2\cdot E_1\\
	\!\!\!& &\!\!\! + \mu^{-1} (\nabla\times\nabla\times E_1\cdot E_2 - \nabla\times\nabla\times E_2\cdot E_1).\notag
\end{eqnarray}
We recall that $E = \Sigma/L$. By some calculations, we get
\begin{equation}
\label{eq:transport}
	\phi\cdot\nabla L + \psi L = 0,
\end{equation}
where
\begin{eqnarray}
	\label{eq:phi}\phi &=& \chi(x)[[(\nabla \mu^{-1}\cdot \Sigma_1)\Sigma_2 - (\nabla \mu^{-1}\cdot \Sigma_2)\Sigma_1] + [(\nabla \Sigma_1)\Sigma_2 - (\nabla \Sigma_2)\Sigma_1]\\
	&& [(\nabla\cdot \Sigma_1)\Sigma_2 - (\nabla\cdot \Sigma_2)\Sigma_1] - 2[(\nabla \Sigma_1)^T\Sigma_2 - (\nabla \Sigma_2)^T\Sigma_1]],\notag\\
	\label{eq:psi} \psi &=& \chi(x)[\nabla\mu^{-1}[(\nabla\times \Sigma_1)\times \Sigma_2 - (\nabla\times \Sigma_2)\times \Sigma_1] \\
	&& + [\nabla(\nabla\cdot \Sigma_1)\Sigma_2 - \nabla(\nabla\cdot \Sigma_2)\Sigma_1] - [\nabla^2 \Sigma_1\cdot \Sigma_2 - \nabla^2 \Sigma_2\cdot \Sigma_1]] \notag.
\end{eqnarray}
Here $\chi(x)$ is any smooth nonzero complex-valued function.

By the method of characteristics, the transport equation
\eqref{eq:transport} has a unique solution only if the integral curves
of $\phi$ connect every point in $\Omega$ to a point on
$\partial\Omega$. We will prove that, with properly chosen CGO
solutions, the integral curves of $\phi$ will be close to straight
lines, and thus connect every internal point to two boundary points.

We make specific choices for $\zeta, a, b$ in \eqref{eq:CGOE}:
\begin{eqnarray}
	\label{eq:zeta} \zeta_1 &=& (1/2, i\sqrt{1/h^2 + 1/4 - k^2}, 1/h),\\
	\notag \zeta_2 &=& (1/2, -i\sqrt{1/h^2 + 1/4 - k^2}, -1/h),
\end{eqnarray}
with $h > 0$ being a free parameter such that $h |\zeta_j| \approx 1$
for $j=1,2$. We note that
\begin{eqnarray}
\label{eq:zeta1}&&\lim_{h\rightarrow 0}\zeta_1/|\zeta_1|=\frac{1}{\sqrt{2}}(0,i,1)=:\zeta_0,\\
\label{eq:zeta2}&&\lim_{h\rightarrow 0}\zeta_2/|\zeta_2| = -\zeta_0
\end{eqnarray}
and
\begin{equation}\notag 
	\zeta_1+\zeta_2 = (1,0,0),\quad \zeta_0\cdot\zeta_0=0.
\end{equation}
By choosing any $a_j\in\CC^3$ and $b_1=b_2=(0,0,1/h)$, we can verify that CGO solutions $\check{E}_1$ and $\check{E}_2$ defined in \eqref{eq:CGOE} satisfy
\begin{eqnarray}
	\notag (\check{E}_0^1 + \check{E}_r^1)\cdot (\check{E}_0^2 + \check{E}_r^2) &=& \check{E}_0^1\cdot \check{E}_0^2 + \mathcal{O}(h)\\
	\notag &=& \frac{k^2}{|\zeta_1|^2|\zeta_2|^2\gamma}(\zeta_1\times b)\cdot(\zeta_2\times b) + \mathcal{O}(h)\\
	\notag &=& \frac{k^2}{|\zeta_1|^2|\zeta_2|^2\gamma} (\frac{1}{h^2}+\frac{1}{2}-k^2)\frac{1}{h^2}  + \mathcal{O}(h)\\
	\notag &=& \frac{k^2}{\gamma}\frac{1}{|\zeta_1|^2|\zeta_2|^2h^4}  + \mathcal{O}(h)
\end{eqnarray}
and
\begin{equation}
	\notag \frac{\zeta_i}{|\zeta_i|}\cdot \check{E}_j =\mathcal{O}(h),\quad \mathrm{for}\; i,j=1,2.
\end{equation}
We recall that the check sign, for example $\check{E_j}$, indicates
fields corresponding with the CGO solutions. We denote
$\check{\Sigma}_j=L\check{E}_j=\vartheta_j e^{i\zeta\cdot
  x}(\eta_j+R_j)$ with $\vartheta = \frac{Lk}{\gamma^{1/2}}$, $\eta_j
= \gamma^{1/2}\check{E}_0^j/k$ and $R=\gamma^{1/2}\check{E}_R/k$. Then
$\eta_1\cdot \eta_2 = 1 + \mathcal{O}(h)$.

We now evaluate $\check\phi$. We find that
\begin{equation}
	\notag \nabla \check{\Sigma}_j = e^{i\zeta\cdot x}[(\eta_j + R_j)(\nabla \vartheta)^T + i \vartheta(\eta_j + R_j) \zeta^T + \vartheta \nabla(\eta_j + R_j)]
\end{equation}
and
\begin{eqnarray}
	\notag &&e^{-i(\zeta_1+\zeta_2)\cdot x}\frac{h}{\vartheta}(\nabla \check{\Sigma}_1)\check{\Sigma}_2 = \nabla (\vartheta E_1)E_2\\
	\notag &=&	h[(\eta_1 + R_1)(\nabla \vartheta)^T + i \vartheta(\eta_1 + R_1) \zeta_1^T + \vartheta \nabla(\eta_1 + R_1)](\eta_2+R_2)\phantom{space}\\
	\notag &=&	h[(\eta_1+R_1)[\nabla\vartheta^T(\eta_2 + R_2) + i\vartheta\zeta_1^T (\eta_2 + R_2)] \\
	\notag && + \vartheta\nabla(\eta_1 + R_1) (\eta_2 + R_2)]\\
	\notag &=& i\vartheta h(\zeta_1\cdot \eta_2)\eta_1 + \mathcal{O}(h)\\
	\notag &\rightarrow& 0\quad \text{as}\;h \rightarrow 0.
\end{eqnarray}
Similarly, on a bounded domain,
\begin{eqnarray}
	\notag &&e^{-i(\zeta_1+\zeta_2)\cdot x}\frac{h}{\vartheta}(\nabla \check{\Sigma}_1)^T\check{\Sigma}_2 \\
	\notag &=&	h[\nabla \vartheta (\eta_1 + R_1)^T(\eta_2 + R_2) + i \vartheta \zeta_1 (\eta_1 + R_1)^T(\eta_2 + R_2)\\
	\notag && + \vartheta \nabla(\eta_1 + R_1)^T (\eta_2 + R_2)]\\
	\notag &=& i\vartheta h\zeta_1 + \mathcal{O}(h)\\
	\notag &\rightarrow& i\vartheta\zeta_0\quad\text{as}\; h \rightarrow 0.
\end{eqnarray}
Moreover,
\begin{eqnarray}
	\notag && e^{-i(\zeta_1+\zeta_2)\cdot x}\frac{h}{\vartheta}(\nabla\cdot \check{\Sigma}_1)\check{\Sigma}_2 \\
	\notag &=& h[(\nabla\vartheta\cdot(\eta_1 + R_1))(\eta_2 + R_2) + i\vartheta(\zeta_1\cdot(\eta_1 + R_1))(\eta_2 + R_2) \\
	\notag && + \vartheta\nabla\cdot(\eta_1 + R_1)(\eta_2 + R_2)]\\
	\notag &=& i\vartheta h(\zeta_1\cdot \eta_1)\eta_2 + \mathcal{O}(h)\\
	\notag &\rightarrow& 0\quad \text{as}\;h \rightarrow 0.
\end{eqnarray}
We also find that
\begin{eqnarray}
	\notag && e^{-i(\zeta_1+\zeta_2)\cdot x}\frac{h}{\vartheta} (\nabla\frac{1}{\mu}\cdot \check{\Sigma}_1)\check{\Sigma}_2 \\
	\notag &=& \vartheta h(\nabla\frac{1}{\mu}\cdot \eta_1)\eta_2 + \mathcal{O}(h)\\
	\notag &\rightarrow& 0\quad \text{as}\;h \rightarrow 0.
\end{eqnarray}

We choose $\chi(x) = -\frac{h}{4}e^{-i(\zeta_1+\zeta_2)\cdot x}$; then
\begin{equation}
\label{eq:phi2}
	\|\check{\phi} - i\vartheta^2 \zeta_0\|_{H^{s-1}_\delta(\RR^3)} = \mathcal{O}(h).
\end{equation}
We let $s=\frac{3}{2}+d+2+\iota$ for $d>0$ and small $\iota>0$. The Sobolev embedding theorem implies that
\begin{equation}
\label{eq:phi5}
	\|\check{\phi} - i\vartheta^2 \zeta_0\|_{C^{d+1}_\delta(\RR^3)} = \mathcal{O}(h).
\end{equation}
If we choose $b_1=b_2=(0,0,h^{\varepsilon-1})$ and $\chi(x) =
-\frac{h^{1-2\varepsilon}}{4}e^{-i(\zeta_1+\zeta_2)\cdot x}$  with
$0<\varepsilon<1/2$, the same calculations as above give
\begin{equation}
\label{eq:phi3}
	\|\check{\phi} - i\vartheta^2 \zeta_0\|_{C^{d+1}_\delta(\RR^3)} = \mathcal{O}(h^{1-2\varepsilon}).
\end{equation}

So far we proved that if the electric fields are exactly given by the
CGO solutions as we constructed, the vector field $\check\phi$ in
\eqref{eq:transport} can be approximated by \eqref{eq:phi2}, which
implies that every integral curve of $\check\phi$ is almost a straight
line and thus connects every internal point to two boundary
points. Next, we prove that we can perturb the CGO
solutions so that an estimate similar to \eqref{eq:phi2} still
holds. To follow the dependencies of vector fields on boundary conditions,
we introduce a regularity theorem for Maxwell's equations. Let $tE$ be
the tangential boundary condition of $E$. The following function
spaces were introduced in \cite{Kenig2011}:
\begin{equation}
	\notag H^l_{\mathrm{Div}}(\Omega) = \{u\in H^l\Omega^1(\Omega): \mathrm{Div}(tu)\in H^{l-1/2}(\partial\Omega)\},
\end{equation}
\begin{equation}
	\notag TH^l_{\mathrm{Div}}(\partial\Omega) = \{g\in H^l\Omega^1(\partial\Omega):\mathrm{Div}(g)\in H^l(\partial\Omega)\},
\end{equation}
for $l>0$, where $H^l\Omega^1(\Omega)$ is a space of vector functions
each component of which is in $H^l(\Omega)$. These are Hilbert spaces
with norms
\begin{equation}
	\notag \|u\|_{H^l_{\mathrm{Div}}(\Omega)} = \|u\|_{H^l(\Omega)} + \|\mathrm{Div}(tu)\|_{H^{l-1/2}(\partial\Omega)},
\end{equation}
\begin{equation}
	\notag \|g\|_{TH^l_{\mathrm{Div}}(\partial\Omega)} = \|g\|_{H^l(\partial\Omega)} + \|\mathrm{Div}(g)\|_{ H^l(\partial\Omega)}.
\end{equation}
It is clear that $t(H^l_{\mathrm{Div}}(\Omega)) = TH^{l-1/2}_{\mathrm{Div}}(\partial\Omega)$. We also observe that 
\begin{equation}\label{eq:reg2}
	\|E\|_{H^{l}(\Omega)}\leq \|E\|_{H^{l}_{Div}(\Omega)} 
	\;\mathrm{and}\; \|G\|_{TH^{l}_{Div}(\partial\Omega)}\leq \|G\|_{H^{l+1}(\partial\Omega)}
\end{equation}

\begin{prop}[\cite{Kenig2011}]\label{thm:regularity}
Let $\epsilon,\mu\in C^l$, $l>2$, be positive functions. There is a
discrete subset $\Sigma\subset\CC$, such that if $\omega\notin\Sigma$,
then one has a solution $E\in H^l_{\mathrm{Div}}$ to \eqref{eq3.1}
given any tangential boundary condition $G\in
TH^{l-1/2}_{\mathrm{Div}}(\partial\Omega)$. The solution satisfies
\begin{equation}
\label{eq:reg}
	\|E\|_{H^l_{\mathrm{Div}}(\Omega)}\leq C\|G\|_{H^{l-1/2}_{\mathrm{Div}}}
\end{equation}
with $C$ independent of $G$.
\end{prop}

We now estimate vector field $\phi$ in \eqref{eq:transport} when the
boundary value of the electric field is given by a perturbation of the
trace of a CGO solution.

\begin{prop}\label{thm:phi}
Under the assumptions of Proposition \ref{thm:regularity}, when $G_j$ is in a neighborhood of $\check{G}_j=t\check{E}_j\in H^{d+3+\iota}(\partial\Omega)$, $j=1,2$, the corresponding vector field $\phi$ defined in \eqref{eq:phi} satisfies
\begin{equation}
\label{eq:phi4}
	\|\phi - i\vartheta^2\zeta_0\|_{C^d(\overline{\Omega})}=\mathcal{O}(h)
\end{equation}
for small $h$.
\end{prop}

Let $s=\frac{3}{2}+d+2+\iota$, for small $\iota>0$. Then this proposition 
follows from Proposition 3.6 in \cite{Chen2013} directly. To be self-contained in this
section, we summarize the proof here.

\begin{proof}
According to the Sobolev embedding theorem, proposition \ref{thm:regularity} and eq. \eqref{eq:reg2}, we have that
\begin{equation}\notag 
	\begin{array}{rl} \vspace{1ex}
	&\|E\|_{C^{d+1}(\overline{\Omega})}\leq C\|E\|_{H^{\frac{5}{2}+d+\iota}(\Omega)}\leq C\|E\|_{H^{\frac{5}{2}+d+\iota}_{Div}(\Omega)}\\
	\leq& C\|G\|_{TH^{d+2+\iota}_{Div}(\partial\Omega)}\leq C\|G\|_{H^{d+3+\iota}(\partial\Omega)},
	\end{array}
\end{equation}
where various constants are all named ``C''. Hence
\begin{equation}
\label{eq:regularity1}
  \|E\|_{C^{d+1}(\overline{\Omega})} \leq C\|G\|_{H^{d+3+\iota}(\partial\Omega)}.
\end{equation}
Let us now define boundary conditions $G_j\in H^{d+3+\iota}(\partial\Omega)$, $j=1,2$, such that
\begin{equation}
\label{eq:regularity2}
  \|G_j - t\check E_j\|_{H^{d+3+\iota}(\partial\Omega)} \leq \varepsilon,
\end{equation}
for some $\varepsilon > 0$ sufficiently small. let $E_j$ be the solution to the Maxwell equations \eqref{eq2.1} with $tE_j = G_j$. By \eqref{eq:regularity1}, we thus have
\begin{equation}
\label{eq:regularity3}
  \|E _j- \check E_j\|_{C^{d+1}(\overline{\Omega})} \leq C\varepsilon,
\end{equation}
for some positive constant $C$. We introduce the complex-valued internal
data, $\Sigma_j=LE_j$ and conclude that
\begin{equation}
\label{eq:regularity4}
  \|\Sigma_j - \check \Sigma_j\|_{C^{d+1}(\overline{\Omega})} \leq C\varepsilon.
\end{equation}
We obtain the estimate (cf. \eqref{eq:phi})
\begin{equation}
\label{eq:regularity5}
	\|\phi-\check{\phi}\|_{C^d(\overline\Omega)} \leq C\|\chi(x)\|_{C^d(\overline{\Omega})}\|\Sigma_j - \check \Sigma_j\|_{C^{d+1}(\overline{\Omega})} \leq C h\varepsilon,
\end{equation}
where $\chi(x)$ is defined above \eqref{eq:phi2}. Therefore,
\eqref{eq:phi4} follows from \eqref{eq:phi5} and
\eqref{eq:regularity5}. This completes the proof.
\end{proof}

We recall that $\mathcal{M}$ is the parameter space of $(L,\gamma)$
defined in \eqref{eq2.5}. We now prove Theorem \ref{thm:unique}.

\begin{proof}[Proof of Theorem \ref{thm:unique}]
Let $d\geq 3$. By proposition \ref{thm:phi}, we choose the set of
boundary conditions for electrical fields to be a neighborhood of
$(\check{G}_j) = (t\check{E}_j)$ in
$H^{d+3+\iota}(\partial\Omega))^2$. By assuming that the measurements coincide, that is,
$\Sigma=\tilde{\Sigma}$, we have that $\phi=\tilde{\phi}$ and $\psi =
\tilde{\psi}$ by \eqref{eq:phi}-\eqref{eq:psi}. Thus, $L$ and
$\tilde{L}$ solve the same transport equation \eqref{eq:transport}
while $L=\tilde{L} = \Sigma/G$ on $\partial\Omega$. As $\phi$
satisfies \eqref{eq:phi4}, we deduce that $L=\tilde{L}$ since integral
curves of $\phi$ map any $x\in\Omega$ to two boundary points. More
precisely, consider the flow $\theta_x(t)$ associated with the imaginary
part of $\phi$, that is, $\theta_x(t)$ is the solution to
\begin{equation}
\label{eq:beta6}\notag 
	\dot{\theta}_x(t) = \Im\phi(\theta_x(t)), \quad \theta(0) = x\in\bar{\Omega}.
\end{equation}
By the Picard-Lindel\"of theorem, \eqref{eq:beta6} admits a unique solution since $\phi$ is of class $C^1(\Omega)$. Also by \eqref{eq:phi4}, for $\forall x\in\Omega$, there exist $x_\pm(x)\in\partial\Omega$ and $t_\pm(x)>0$ such that
\begin{equation}
\label{eq:beta7}\notag 
	\theta_x(t_\pm(x)) = x_\pm(x) \in\partial\Omega.
\end{equation}
By the method of characteristics, the solution $L$ to the transport equation \eqref{eq:transport} is given by
\begin{equation}
\label{eq:beta8}
	L(x) = L_0(x_\pm(x))e^{-\int_0^{t_\pm(x)}\Im\gamma(\theta_x(s))ds} ,
\end{equation}
where $L_{0}:=L|_{\partial\Omega}$ is the restriction of $L$ on the
boundary. The solution $\tilde{L}$ is given by the same formula since
$\theta_x(t) = \tilde{\theta}_x(t)$. This implies that $L=\tilde{L}$ and
thus $E_j = \tilde{E}_j = \Sigma_j/L = \tilde{\Sigma}_j/\tilde{L}$,
$j=1,2$. By the choice of illuminations, we also have $|E_j|\neq 0$
due to \eqref{eq:regularity3} and $|\check{E}_j|\neq 0$. Finally, by
substituting $E_j$ and $\tilde{E}_j$ into \eqref{eq2.1}, we can solve
for $\gamma$ and $\tilde{\gamma}$ by
\begin{equation}
\label{eq:gamma1}
	\gamma=-\displaystyle\frac{(\Sigma\wedge\frac{1}{\mu}\Sigma\wedge E)\cdot\bar{E}}{\omega^2|E|^{2}}\quad\mathrm{and}\quad \tilde{\gamma}=-\displaystyle\frac{(\Sigma\wedge\frac{1}{\mu}\Sigma\wedge \tilde{E})\cdot\bar{\tilde{E}}}{\omega^2|\tilde{E}|^{2}}.
\end{equation}
We then conclude that $\gamma=\tilde{\gamma}$ in $C^{d-3}(\Omega)$ for $d\geq 3$.
\end{proof}


\subsection{Stability result}\label{se:stab}

The proof of the stability theorem follows the argument of
\cite{Chen2013} with the estimate of the vector field replaced by
\eqref{eq:phi4}.

\begin{prop}\label{thm:stab2}
Let $d\geq 1$. Let $L$ and $\tilde{L}$ be solutions to
\eqref{eq:transport} corresponding to coefficients $(\phi,\psi)$ and
$(\tilde{\phi}, \tilde{\psi})$, respectively, where \eqref{eq:phi4}
holds for both $\phi$ and $\tilde{\phi}$.

Let $L_0 = L|_{\partial\Omega}$, $\tilde{L}_0 =
\tilde{L}|_{\partial\Omega}$, and $L_0,\tilde{L}_0 \in
C^{d-1}(\partial\Omega)$. We also assume that $h$ is sufficiently small and that
$\Omega$ is convex. Then there is a constant $C$ such that upon restricting
to $\Omega_1$,
\begin{align}
\label{eq:stable1}
\|L-\tilde{L}\|_{C^{d-1}(\overline{\Omega}_1)} \leq & C\|L_0\|_{C^{d-1}(\partial\Omega_1)}[\|\phi-\tilde{\phi}\|_{C^{d}(\overline{\Omega}_1)} \\
& + \|\psi-\tilde{\psi}\|_{C^{d-1}(\overline{\Omega}_1)}] + C\|L_0-\tilde{L}_0\|_{C^{d-1}(\partial\Omega_1)}. \notag
\end{align}
\end{prop}

The choice of $\Omega_1$ depends on the proof of the above
proposition. Readers are referred to \cite{Chen2013} for the
details. Now we can prove the main stability theorem.

\begin{proof}[Proof of Theorem \ref{thm:stability1}]
From \eqref{eq:phi} and \eqref{eq:psi}, it is straightforward to check that
\begin{eqnarray}
	\notag \|\phi-\tilde{\phi}\|_{C^{d}(\overline{\Omega}_1)}&\leq& C \|\Sigma-\tilde{\Sigma}\|_{C^{d+1}(\overline{\Omega}_1)}\;\mathrm{and}\;\\
	\notag \|\psi-\tilde{\psi}\|_{C^{d-1}(\overline{\Omega}_1)}&\leq& C \|\Sigma-\tilde{\Sigma}\|_{C^{d+1}(\overline{\Omega}_1)},
\end{eqnarray}
where $C>0$ is a positive constant. The first part of \eqref{eq2.6}
then follows directly from Proposition \ref{thm:stab2}. To estimate
the difference between $\gamma$ and $\tilde{\gamma}$, we notice that
$$E-\tilde{E}=\displaystyle\frac{\Sigma}{L}-\displaystyle\frac{\tilde{\Sigma}}{\tilde{L}}=\displaystyle\frac{L(\Sigma-\tilde{\Sigma})-\Sigma(L-\tilde{L})}{L\tilde{L}}.$$
Since $L$ and $\tilde{L}$ are non-vanishing, by the stability result
for $L$ we obtain
\begin{equation}\label{Estability}
\|E-\tilde{E}\|_{C^{d-1}(\overline{\Omega}_{1})}\leq C\|\Sigma-\tilde{\Sigma}\|_{C^{d+1}(\overline{\Omega}_{1})}.
\end{equation}
By choosing the boundary values close to the boundary conditions of CGO solutions, \eqref{eq:regularity2} and \eqref{eq:regularity3} imply that $E_j$ is non-vanishing since the CGO solutions are non-vanishing. We recall that $\gamma$ and $\tilde{\gamma}$ are computed by \eqref{eq:gamma1}. By taking the difference and using \eqref{Estability} we derive
$$
	\|\gamma-\tilde{\gamma}\|_{C^{d-3}(\overline{\Omega}_{1})}\leq C\|\Sigma-\tilde{\Sigma}\|_{C^{d+1}(\overline{\Omega}_{1})}.
$$
This completes the proof.
\end{proof}

The required regularity of the different variables is summarized in
Table \ref{tab}, where $s=\frac{3}{2}+d+2+\iota$. 

\begin{table}[bth]
	\caption{Regularity of fields and coefficients}
	\label{tab}
	\centering
	\begin{tabular}{lll}
		\hline\hline &&\\[-8pt]
		${\mu},{\gamma}$ & $H^{s+2}(\RR^3)$ & Required regularity for the construction of CGO's\\
		${Q_{\mu,\gamma}},Z$ & $H^s(\RR^3)$& $s\geq \frac{3}{2}$ and $Q_{\mu,\gamma}$ is compactly supported \\
		$Y,X,\check{E},\check{\Sigma}$ & $H^{s-1}(\RR^3)$ & Constructed regularity\\
		$G$ & $H^{s-\frac{1}{2}}(\partial\Omega)$ & Required regularity for boundary source\\
		$L,\Sigma$ & $H^{s-1}(\overline{\Omega})$ & Required regularity for inversion, $H^{s-1}(\overline{\Omega})\subset C^{d+1}(\overline{\Omega})$\\
		$\phi$ & $C^d(\overline{\Omega})$ & \\
		$\psi,L,E$ & $C^{d-1}(\overline{\Omega})$ & Recovered regularity, $d\geq 1$\\
		$\gamma$ & $C^{d-3}(\overline{\Omega})$ & Recovered regularity, $d\geq 3$\\
		\hline\hline
	\end{tabular}	
\end{table}

\subsection{Stability with 6 complex internal data}\label{se:stab2}

We now follow Section 3.4 of \cite{Chen2013} to prove Theorem \ref{thm:stability2}.
If we take more internal measurements, we can rewrite
\eqref{eq:transport} into matrix form. We first construct proper CGO
solutions. Let $j=1,2,3$ in this section. We can choose unit vectors
$\zeta_0^j$, such that $\zeta_0^j\cdot \zeta_0^j = 0$ and
$\{\zeta_0^j\}$ are linearly independent. We next choose $(\zeta_1^j,
\zeta_2^j)$ such that $|\zeta|:=|\zeta_1^j|=|\zeta_2^j|$ and
\begin{align}\label{eq:stable4}\notag 
\lim_{|\zeta|\rightarrow\infty} \frac{\zeta_1^j}{|\zeta|} = \lim_{|\zeta|\rightarrow\infty} \frac{\zeta_2^j}{|\zeta|} = \zeta_0^j.
\end{align}
$(a_1^j,a_2^j)$ and $(b_1^j,b_2)$ are chosen such that $a_1^j=a_2^j$, $b_1^j=b_2^j$,
\begin{equation}
	\notag a_1^j\cdot\zeta_0^j=0\quad\mathrm{and}\quad 0<|\zeta_1^j\times b_1^j| =\mathcal{O}(|\zeta_1|^2).
\end{equation}

We construct CGO solutions $\check{E}_1^j, \check{E}_2^j$ corresponding to $(\zeta_1^j,a_1^j,b_1^j)$ and $(\zeta_2^j, a_2^j,b_2^j)$, for $j=1,2,3$. Let the boundary illuminations $G_1^j, G_2^j$ be chosen according to \eqref{eq:regularity2}. The measured internal data are then given by $D_1^j, D_2^j$. Proposition \ref{thm:regularity} shows that the vector field defined by \eqref{eq:phi} satisfies 

\begin{equation}\label{eq:stable5}\notag 
	\|\phi^j - i\vartheta^2\zeta_0^j\|_{C^d(\overline{\Omega})} \leq \frac{C}{|\zeta|}.
\end{equation}

The rest of the proof of Theorem \ref{thm:stability2} then follows the argument in Section 3.4 of \cite{Chen2013}.

\section{The temporal behavior of CGO solutions}\label{se:time}

Here, we characterize the internal data following from the CGO
solutions introduced above, in particular, in the small $h$ limit. We
find that their frequency behavior is approximately Gaussian.

Let $\zeta_1$ and $\zeta_2$ be defined in \eqref{eq:zeta}. We notice
that, for small $h$,
\begin{equation}
	\label{eq16} \sqrt{\frac{1}{h^2}+\frac{1}{4}-k^2} = \sqrt{\frac{1}{h^2}+\frac{1}{4}} -\frac{hk^2}{2\sqrt{1+\frac{h^2}{4}}} + \mathcal{O}(h^3k^4).
\end{equation}
We let
\begin{eqnarray}
	\notag\tilde{\zeta}_1 &=& \left(\frac{1}{2}, \frac{1}{h}, \mathrm{i}\sqrt{\frac{1}{h^2}+\frac{1}{4}}\right)\\
	\notag\tilde{\zeta}_2 &=& \left(\frac{1}{2}, -\frac{1}{h}, -\mathrm{i}\sqrt{\frac{1}{h^2}+\frac{1}{4}}\right).
\end{eqnarray}

To generate CGO solutions with Gaussian behavior in frequency, we
consider two choices of parameters $(a_j,b_j)$, $j=1,2$:\\[0.3cm]
\textbf{(1)} We choose $a_j\in\CC^3$ such that $a_j\cdot\zeta_j=0$ for
$j=1,2$, and $b_1=b_2=(0,0,1/h)$. Then $\check{E}_{0,j} =
\mathrm{i}\omega\tilde{E}_{0,j} + \mathcal{O}(h^2\omega^2)$, where
$\tilde{E}_{0,j}$ is independent of $\omega$, 
\begin{equation*}
	\tilde{E}_{0,j} = \frac{\mu_0\gamma_0\sqrt{\gamma}}{\mathrm{i}|\zeta|^2}(\tilde{\zeta}_j\times b) = \mathrm{i}\frac{\mu_0\gamma_0\sqrt{\gamma}}{|\zeta|^2h^2} \left((-1)^j, \frac{h}{2}, 0\right).
\end{equation*}
Assuming that $\omega$ is bounded, we have
\begin{equation}
\label{eq20}
\begin{array}{rcl}
	\displaystyle \check E_1(x,\omega) 
	&=& \displaystyle \mathrm{e}^{-\alpha h\omega^2}
	\mathrm{e}^{\mathrm{i}\tilde{\zeta}_1\cdot x}(\mathrm{i}
	\omega\tilde{E}_{0,1} + \mathcal{O}(h)),\\[6pt]
	\displaystyle \check E_2(x,\omega) 
	&=& \displaystyle \mathrm{e}^{ \alpha h\omega^2}
	\mathrm{e}^{\mathrm{i}\tilde{\zeta}_2\cdot x}(\mathrm{i}
	\omega\tilde{E}_{0,2} + \mathcal{O}(h)),
\end{array}
\end{equation}
where
\begin{equation*}
	\alpha  = \frac{\mu_0\gamma_0x_3}{2\sqrt{1+\frac{h^2}{4}}}.
\end{equation*}	

\medskip\medskip

\noindent
\textbf{(2)} We choose $a_j$ such that $a_j\cdot\zeta_j=1/h$ and
$b_1=b_2=(0,0,h^{\varepsilon-1})$ for $0<\varepsilon<1/2$. Then
$E_{0,j} = \tilde{E}_{0,j} + \mathcal{O}(i\omega h^\varepsilon)$, where
\begin{equation}
	\notag \tilde{E}_{0,j} = \frac{\sqrt{\gamma}}{|\zeta_j|^2h}\tilde{\zeta}_j.
\end{equation}
Then the CGO solutions are
\begin{equation}
\label{eq21}
\begin{array}{rcl}
	\displaystyle \check E_1(x,\omega) 	&=& \displaystyle \mathrm{e}^{-\alpha h\omega^2}\mathrm{e}^{\mathrm{i}\tilde{\zeta}_1\cdot x}(\tilde{E}_{0,1}(x) + \mathcal{O}(\mathrm{i}\omega h^\varepsilon)),\\[6pt]
	\displaystyle \check E_2(x,\omega) &=& \displaystyle \mathrm{e}^{ \alpha h\omega^2}\mathrm{e}^{\mathrm{i}\tilde{\zeta}_2\cdot x}(\tilde{E}_{0,2}(x) + \mathcal{O}(\mathrm{i}\omega h^\varepsilon)).
\end{array}
\end{equation}	

\noindent
\textbf{Remark}: Although the choice of $b_j$ such that $\zeta_j\times
b_j=0$ could simplify the above argument, it will destroy the proof in
Section \ref{se:unique}. Also, in the second choice of $(a_j,b_j)$, we
would prefer large $\varepsilon$ to get better approximations in
\eqref{eq21}. However, the proof in Section \ref{se:unique} requires
larger $1-2\varepsilon$ in \eqref{eq:phi3}, that is, smaller $\varepsilon$, to obtain a
better stability estimate. We need to choose $\varepsilon$ to balance
these two conditions.

\medskip\medskip

\noindent
Suppose $\omega$ is bounded. Upon taking an inverse Fourier
transform of \eqref{eq20} and \eqref{eq21}, we obtain 
\[
   \mathcal{F}^{-1} \check E_j(x,t) \approx f(t)\mathrm{e}^{\mathrm{i} \tilde{\zeta}_2\cdot x} 
   \tilde{E}_{0,j}(x), \quad\mathrm{for}\quad j=1,2,
\]
where $f(t)$ is a function in time and concentrates at zero time while $h$ is small. The above equation also characterizes the temporal behavior of electric source $G$ and the internal data $\Sigma$. In experiments, one can generate an electric source with particular waveform by convolving $G$ with a window function, $\hat{W}$ say, centered at a particular frequency corresponding with the certain waveform in time. 


\section{Acknowledgment}\label{se:acknowledge}

J. Chen was funded by Total S.A.

\end{document}